\documentclass[a4paper, 12pt]{amsart}

\usepackage[utf8]{inputenc}

\usepackage[T2A]{fontenc}
\usepackage{amsmath,amssymb,amsthm}
\usepackage[a4paper,hmargin=2.5cm,vmargin=2.5cm]{geometry}
\usepackage[english]{babel}

\usepackage{hyperref}

\usepackage{amsthm}
\usepackage{amscd}
\usepackage{amsmath}
\usepackage{amsfonts}
\usepackage{amssymb}
\usepackage{graphicx}
\usepackage{amsopn}
\usepackage[usenames]{color}
\usepackage{tikz}
\usepackage{tikz-cd}
\usepackage{textcomp}
\usetikzlibrary{arrows,shapes,snakes,automata,backgrounds,petri}

\theoremstyle{theorem}
\newtheorem{theorem}{Theorem}
\newtheorem{corollary}[theorem]{Corollary}
\newtheorem{prop}[theorem]{Proposition}

\newtheorem{lemma}[theorem]{Lemma}

\theoremstyle{definition}
\newtheorem{definition}[theorem]{Definition}

\newtheorem{remark}[theorem]{Remark}

\theoremstyle{definition}
\newtheorem*{Notation}{Notation}

\DeclareMathOperator{\inv}{inv}
\DeclareMathOperator{\sym}{Sym}
\DeclareMathOperator{\res}{Res}

\def\C{\mathbb{C}}
\def\eps{\varepsilon}

\numberwithin{theorem}{section}
\allowdisplaybreaks[3]

\begin{document}
\date{}
\author{Konstantin M. Posadskiy}
\title[Continuous two-valued dynamical systems]{Continuous two-valued discrete-time dynamical systems and actions of two-valued groups}
\address{
    HSE University
}
\email{kmposadskiy@hse.ru}
\maketitle{}

\begin{abstract}
We study continuous 2-valued dynamical systems with discrete time (dynamics) on $\mathbb{C}$. The main question addressed is whether a 2-valued dynamics can be defined by the action of a 2-valued group. We construct a class of strongly invertible continuous 2-valued dynamics on $\mathbb{C}$ such that none of these dynamics can be given by the action of any 2-valued group. We also construct an example of a continuous 2-valued dynamics on $\mathbb{C}$ that is not strongly invertible but can be defined by the action of a 2-valued group.
\end{abstract}

\section{Introduction}

\subsection{Main definitions}\

\begin{definition}
    Let $S$ be a set and let $\sym^m(S)$ be the $m$-th symmetric power of $S$.
    A map $T: S \to \sym^m(S)$ is called an \textit{$m$-valued dynamical system with discrete time} (\textit{$m$-valued dynamics}) on the set $S$.
\end{definition}

For any $m$-valued dynamics $T$ on a set $S$ we can construct a directed graph $G$ with the set of vertices equal to $S$ and the multiset of edges containing a pair $(s_1, s_2) \in S^2$ as many times as $s_2$ lies in the multiset $T(s_1)$.
We call an $m$-valued dynamics $T$ \textit{weakly invertible} if for each vertex $s$ of $G$ there exists an incoming edge $(*, s)$; we call an $m$-valued dynamics $T$ \textit{strongly invertible} if for each vertex $s$ there exist exactly $m$ such edges.
We use square brackets to enumerate elements of a multiset.

\begin{remark}
    A map $f: S \longrightarrow S$ acts naturally on $\sym^m(S)$. A multiset $[s_1, \dots, s_n]$ maps to $[f(s_1), \dots, f(s_n)]$.
\end{remark}

The definition of a continuous $m$-valued dynamics is derived naturally.

\begin{definition}
    Let $S$ be a topological space. Then $S^m$ and therefore $\sym^m(S)$ also have a natural topological space structure. A continuous map $T: S \to \sym^m(S)$ is called a \textit{continuous $m$-valued dynamics}. 
    We denote the set of continuous $m$-valued dynamics on $S$ by $\mathcal{T}_m(S)$.
\end{definition}

The definitions of $m$-valued group and its action on a set were given by V.M. Buchstaber (see \cite{mval}). 
We repeat these definitions here.

An \text{$m$-valued multiplication} on a set $X$ is a map $$\mu: X \times X \to \sym^m(X)$$
Let us use the notation $\mu(x, y) = x * y$. We have the following natural generalizations of the standard axioms of group multiplication.

\textit{Associativity}: The multisets $[x * (y * z)]$ and $[(x * y) * z]$ consisting of $m^2$ elements are equal.

\textit{Unit}: An element $e \in X$ such that $e * x = x * e = [x, x, \dots, x]$ for all $x \in X$.

\textit{Inverse}: A map $\inv: X \to X$ such that $e \in \inv(x) * x$ and $e \in x * \inv(x)$ for all $x \in X$.

\begin{definition}
    The map $\mu : X \times X \to \sym^m(X)$ defines an \textit{$m$-valued group structure} $\mathcal{X} = (X, \mu, e, \inv)$ on $X$ if it is associative, has a unit and an inverse. In this case we simply say that $X$ is an $m$-valued group.
\end{definition}

\begin{definition}
    Let $X$ be a $m$-valued group. A subset $Y \subset X$ is called an \textit{$m$-valued subgroup of group $X$ generated by an element $a \in X$} if it is the minimal-inclusion subset with following properties:

    1) $a \in Y$

    2) $\forall b \in Y$ we have $\inv (b) \in Y$.

    3) For all $b, c \in Y$ the set $Y$ contains every element of the multiset $b*c$.

    An $m$-valued group $X$ is called \textit{single-generated with generator $a$}, if $X = Y$.
\end{definition}

The following definition is central to this paper. It connects the concepts of a multi-valued group and a multi-valued dynamics on a set.
\begin{definition} \label{definition:action}
    An $m$-valued group $A$ with unit $e$ and multiplication $\mu$ \textit{acts on a space} $S$ if there is a mapping $\nu: A \times S \to \sym^m(S)$ such that
    
    1) the two multisets $\nu(a_1, \nu(a_2, s))$ and $\nu(\mu(a_1, a_2), s)$ consisting of $m^2$ elements are equal for all $a_1, a_2 \in A$, $s \in S$
    
    2) $\nu(e, s) = [s, \dots, s]$ for all $s \in S$.

    We say that an $m$-valued dynamics $T \in \mathcal{T}_m(S)$ \textit{is defined by the action of a 2-valued group $A$ with an element $a$} if there exists an action $\nu$ of $A$ on $S$ such that for all $s \in S$ the multisets $T(s)$ and $\nu(a, s)$ are equal.
\end{definition}

\begin{remark}
    Note that an $m$-valued dynamics is given by the action of an $m$-valued group $A$ with an element $a$ if and only if it is also given by the action of the subgroup $\langle a \rangle \subset A$ generated by the element $a$. Therefore, it does not make any difference whether we consider the $m$-valued group $A$ or its subgroup $\langle a \rangle$. For the sake of simplicity, we consider the whole $m$-valued group $A$ in this paper.
\end{remark}

\subsection{Problem statement and results}\

The question of whether a multivalued dynamics can be defined by the action of a multivalued group is partly motivated by the problem of the growth in the number of images of a single point under iterations of a multivalued dynamics (see \cite{VES}). It is also related to the question of what should be properly understood as the integrability of a multivalued dynamics.

The further investigations into the integrability of multivalued dynamics have been carried out in the works \cite{B-G}, \cite{G-YA}, and \cite{B-V}.

A one-valued dynamics $T$ is a map from $S$ to $S$, so in this case:

    1) Any one-valued dynamics is defined by the action of the semigroup $\mathbb{Z}_{\geqslant 0}$.

    2) If a one-valued dynamics is invertible ($\forall y \; \exists ! \: x: T(x) = y)$, then it is defined by the action of the group $\mathbb{Z}$.

The question arises whether there are analogues of these statements for $m$-valued dynamics.

A. Gaifullin and P. Yagodovskii studied discrete dynamics in the paper \cite{G-YA}. They obtained a partial answer, namely, for an arbitrary strongly invertible $m$-valued dynamics, they described a method for constructing an $m$-valued group whose action defines this dynamics. The non strongly invertible case was not addressed, and it was assumed that in this case the $m$-valued dynamics could not be defined by the action of an $m$-valued group.
\begin{theorem} \label{theorem:irreversible}
    There exists a non strongly invertible continuous 2-valued dynamics such that this dynamics is defined by the action of some 2-valued group.
\end{theorem}

The construction described in \cite{G-YA} does not provide an answer to the question of whether a given continuous multivalued dynamics can be defined by the action of a multivalued group, even for strongly invertible continuous multivalued dynamics. This is because the action of the multivalued group obtained through this construction is generally not continuous. In Section 2 we investigate whether continuous 2-valued dynamics of the following form can be defined by the action of a 2-valued group:

Take a polynomial $$P(z, w) = w^m + q_{m-1}(z)w^{m-1} + \dots + q_0(z)$$ 
for some $q_{m-1}, \dots, q_0 \in \C[z]$.
If we fix $z$ the polynomial $P(z, w)$ becomes a polynomial of a single variable. Then we have an $m$-valued map such that $z$ maps to the multiset $[w_1, \dots, w_m]$ of roots of $P_z(w)$. This $m$-valued map is continuous.
We denote this class of continuous $m$-valued dynamics by $\mathcal{P}_m(\C)$.

If $m=2$, then the polynomial $P_z(w)$ has a form $w^2 + 2p_1(z)w + p_0(z)$. Then $z \mapsto [w_1, w_2]$, namely,
\begin{equation*}
    z \mapsto -p_1(z) \pm \sqrt{p_1^2(z) - p_0(z)}
\end{equation*}

The expression under the radical is a complex number; in this paper we denote the pair of complex numbers whose squares are equal to $c$ by $\pm \sqrt{c}$. In the paper we study non-degenerate 2-valued dynamics $T$:

\begin{definition}
    We call a 2-valued dynamics $T \in \mathcal{P}_2(\C)$ \textit{non-degenerate} if $T$ cannot be represented as a composition of a mapping $\C \to \C^2$ and the projection $\C^2 \to \sym^2(\C)$ and there exists a point $z \in \C$ such that the multiset $T(T(z))$ consists of four distinct elements.
\end{definition}

The main result of the second section of the paper is a necessary condition for a non-degenerate 2-valued dynamics $T \in \mathcal{P}_2(\C)$ to be defined by the action of a 2-valued group.

\begin{theorem} \label{theorem:criteria}
    Let $T$ be a non-degenerate 2-valued dynamics in $\mathcal{P}_2(\C)$: 
    $$T(z) = -p_1(z) \pm \sqrt{p_1^2(z) - p_0(z)}$$
    Then a necessary (but not sufficient) condition for $T$ to be defined by the action of a 2-valued group is that the polynomial $p_0$ be a perfect square.
\end{theorem}

This theorem provides a source of 2-valued dynamics that are not defined by the action of a 2-valued group. Moreover, unlike the discrete case almost no continuous 2-valued dynamics are defined by the action of a 2-valued group. Among these dynamics many are strongly invertible.

\begin{corollary}
    If the polynomial $p_1$ is linear and $p_0$ is a polynomial of degree 2 with distinct roots, then the corresponding 2-valued dynamics from $\mathcal{P}_2(\C)$ defined by the polynomial $P(z, w)$ is strongly invertible but cannot be defined by the action of a 2-valued group.
\end{corollary}

\subsection{Acknowledgements}\ 

The author is deeply grateful to his scientific advisor A. A. Gaifullin for useful discussions, constant attention, and invaluable advice that significantly improved this work. The author also sincerely thanks M. T. Urmanov and Yu. V. Chekanov for their thoughtful suggestions and contributions to the paper.

\section{Two-valued dynamics that cannot be defined by the action of a two-valued group}

In this section we study non-degenerate 2-valued dynamics $T$ of the form
$$z \mapsto -p_1(z) \pm \sqrt{p_1^2(z) - p_0(z)}$$

By applying a conjugation by the shift $z \mapsto z + a$, we can ensure that the expression under the radical has a root of odd multiplicity at $0$. If the expression under the radical is a perfect square, then the dynamics is degenerate.

\subsection{Double application of the dynamics $T$}\

\begin{lemma}
$T \circ T$ as a 4-valued dynamics is defined by a polynomial of 4-th degree: $z \mapsto [v_1, v_2, v_3, v_4]$, which is the multiset of the roots of some polynomial $$v^4 + q_3(z)v^3 + q_2(z)v^2 + q_1(z)v + q_0(z).$$
\end{lemma}

The proof immediately follows by eliminating $w$ from the system of equations
$$v^2 + 2p_1(w)v + p_0(w) = 0$$
$$w^2 + 2p_1(z)w + p_0(z) = 0$$
using the resultant.

The polynomial $P_z(v) = v^4 + q_3(z)v^3 + q_2(z)v^2 + q_1(z)v + q_0(z)$ has roots of multiplicity greater then 1 if and only if $\res (P_z, P'_z) = 0$. The resultant $\res(P_z, P'_z)$ is a polynomial of $z$, therefore the set of points $z$ that map to four distinct points is either empty or coincides with $\C$ minus a finite set of points. This implies the following proposition:

\begin{prop} \label{lemma:4seasons}
    If the dynamics $T$ is non-degenerate, then $T(T(z))$ consists of four distinct points for all $z \in \C$ except for a finite number of points.
\end{prop}

\subsection{Images of simple closed curves}\

\begin{Notation} \label{notation}
Denote by $z_1$ an arbitrary root of the polynomial $p_0(z)$. Then the dynamics $T$ takes $z_1$ to the pair $[0, -2p_1(z_1)]$. Also, since $0$ is the root of $p_1^2(z) - p_0(z)$, it follows that the two images of $0$ coinside. Let us introduce the following notation (see the figure below):

\begin{align*}
T(z_1) &= [0, z_2],\\
T(0) &= [z_0, z_0],\\
T(z_0) &= [z_3, z'_3],\\
T(z_2) &= [z_4, z'_4]
\end{align*}

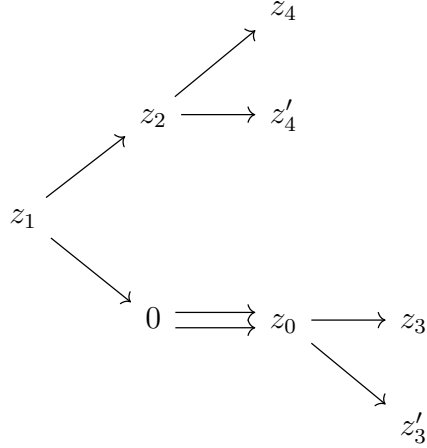
\[
\begin{tikzcd}
    & & z_4 \\
    & z_2 \ar[ru] \ar[r] & z'_4 \\
    z_1 \ar[ru] \ar[rd] \\
    & 0 \ar[r, shift left=1] \ar[r, shift right=1] & z_0 \ar[r] \ar[rd] & z_3 \\
    & & & z'_3
\end{tikzcd}
\]
\end{Notation}

\begin{remark}
    Some of the points $z_1, 0, z_2, z_3, z'_3, z_4, z'_4$ may coincide.
\end{remark}

We study 2-valued dynamics by examining the images of simple closed curves around different points in their small neighborhoods. Outside the diagonal, the projection
$$\C^2 \to \sym^2(\C)$$
is a two-sheeted covering of the set $\sym^2(\C) \setminus \{[x, x] \; | \; x \in \C\}$. We use the proposition that follows from path lifting property:
\begin{prop} \label{statement:branches}
    Let $\gamma$ be a path on $\C$ such that $T(\gamma)$ doesn't contain pairs of the form $[x,x]$. Then the action of the 2-valued dynamics $T$ on $\gamma$ has two continuous branches.
\end{prop}

The set of points whose images under the action of $T$ lie on the diagonal is finite, and thus all such points are isolated. Therefore, it follows from Proposition \ref{statement:branches} that the image of a simple closed curve around any point in a small neighborhood of this point under the action of a non-degenerate 2-valued dynamics is either a pair of closed paths or a pair of paths where the end of each path is the beginning of the other.

Denote a simple closed curve around $z_1$ in a small neigbourhood of this point by $\gamma_1$.

Recall that $z_1$ is a root of $p_0(z)$.

\begin{prop} \label{lemma:gamma1}
    If $z_1$ is a root of the polynomial $p_0$ of odd multiplicity, then one of the following two situations holds:

    1) the image of $\gamma_1$ under the action of the dynamics $T$ is a pair of closed paths, at least one of these paths makes an odd number of turns around zero

    2) the image of $\gamma_1$ is a closed curve. This curve makes an odd number of turns around zero, the pair of images of a point lying on $\gamma_1$ is swapped when this point traverses the curve $\gamma_1$ once.
\end{prop}

\begin{proof}
If $z$ is not a root of $p_1$, then $\gamma_1$ is mapped to a pair of closed paths near $0$ and near $z_2 = -2p_1(z)$ under a single application of the dynamics. We denote these paths by $\omega$ and $\gamma_2$ respectively. Since $\gamma_2$ lies in a neighborhood of the point $z_2$ and thus turns around $0$ zero times, it follows that $\omega$ turns around $0$ the same number of times as the image of $\gamma_1$ under the mapping
\begin{equation} \label{p_0}
    z \mapsto \left(-p_1(z) + \sqrt{p_1^2(z) - p_0(z)}\right)\left(-p_1(z) - \sqrt{p_1^2(z) - p_0(z)}\right) = p_0(z)
\end{equation}
Therefore the number of turns $\omega$ makes around zero equals the multiplicity of $z_1$ as a root of the polynomial $p_0$ and this multiplicity is odd.

If $z_1$ is a root of $p_1$, then $T \circ T (z_1) = [0, 0]$. The dynamics $T$ in this case takes the form
$$(z-z_1)^a \tilde p_1(z) \pm \sqrt{(z-z_1)^{2a} \tilde p_1(z)^2 - (z-z_1)^b \tilde p_0(z)}$$
We know that $b$ is odd, because the root $z_1$ of the polynomial $p_0$ has odd multiplicity. Therefore there are two cases.
If $2a > b$, then $z_1$ is a root of odd multiplicity of 
$$(z-z_1)^{2a} \tilde p_1(z)^2 - (z-z_1)^b \tilde p_0(z)$$
When the point $g$ traverses $\gamma_1$, the element $p_1(g)$ returns to its original position, while $\sqrt{p_1^2(z) - p_0(z)}$ changes sign. Therefore, the image of $\gamma_1$ in this case is a curve that makes $b$ turns around zero, and the two images of a point lying on $\gamma_1$ are swapped when $\gamma_1$ is traversed once.
If $2a < b$, then both images of a point lying on $\gamma_1$ return to their original positions when this point traverses $\gamma_1$. This means that the image of $\gamma_1$ is a pair of curves in a neighborhood of zero.
It follows from equality \eqref{p_0} that these two curves together make the same number of turns around zero as $p_0$, namely $b$ turns. Since $b$ is odd, it follows that one of these curves makes an odd number of turns around zero.
\end{proof}

Now consider a closed curve $\omega$ in a neighborhood of 0, looping around 0 an odd number $d$ of times. Fix a point $g_1$ on $\omega$. Denote the images of $g_1$ under the action of $T$ by $g_{11}, g_{12}$.
\begin{prop} \label{lemma:omega}
    When the point $g_1$ traverses the curve $\omega$, its images $g_{11}, g_{12}$ under the action of $T$ swap.
\end{prop}

\begin{proof}
Let us examine how the image of the point $g_1$ changes when traversing $\omega$. Consider one of the branches. There, the point $g_1$ maps to the point $g_{11}$. When traversing $\omega$, the image of $g_1$ under the action $z \mapsto -p_1(z)$ returns to itself. Recall that $0$ is a root of $p_1^2(z) - p_0(z)$ of odd multiplicity. Therefore, this expression takes the form $z^{2k+1} \cdot  q(z)$, where $q(z)$ is a polynomial with a nonzero constant term $q_0$. In a sufficiently small neighborhood of zero, the higher-order terms are negligible, and $\sqrt{z^{2k+1} \cdot  q_0}$ changes sign when the point $z$ traverses the curve $\omega$, since the increment in the argument of the complex number is $\pi d (2k + 1)$, which corresponds to a half-integer number of turns. This implies that the images of $g_1$ are swapped when traversing $\omega$.
\end{proof}

\subsection{Defining a 2-valued dynamics by the action of a 2-valued group}\

Let a 2-valued dynamics $T$ be defined by the action of a 2-valued group $A$ with an element $a$: $\nu(a, c) = T(c)$ for all $c \in \C$. 

Denote $a*a$ by $[a_1, a_2]$.
Denote the 2-valued dynamics $\nu(a_1, c)$ by $T_1(c)$, the 2-valued dynamics $\nu(a_1, c)$ by $T_2(c)$. Consider the 4-valued dynamics $T \circ T$. For all $c \in C$ we have
$$T(T(c)) = \nu(a, \nu(a, c)) = \nu(a*a, c) = [T_1(c), T_2(c)],$$
and therefore
\begin{prop} \label{proposition:4div2}
The dynamics $T$ can be defined by the action of a 2-valued group only if 4-valued dynamics $T \circ T$ splits into two continuous 2-valued dynamics.
\end{prop}

Denote the images of $c$ under the action of $T$ by $c_1, c_2$, the images of $c_1$ by $c_{11}, c_{12}$, and the images of $c_2$ by $c_{21}, c_{22}$. Suppose that all four points $c_{11}, c_{12}, c_{21}, c_{22}$ are distinct.

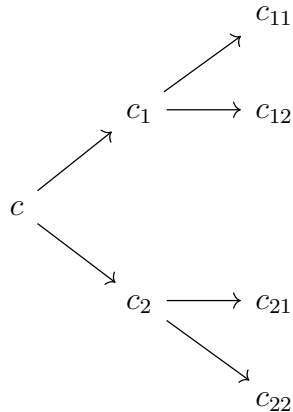
\[
\begin{tikzcd}
    & & c_{11} \\
    & c_{1} \ar[ru] \ar[r] & c_{12} \\
    c \ar[ru] \ar[rd] \\
    & c_2 \ar[r] \ar[rd] & c_{21} \\
    & & c_{22}
\end{tikzcd}
\]

Any two arrows originating from the same vertex are equivalent. Accordingly, there are two fundamentally distinct ways to split the quadruple of images $c_{11}, c_{12}, c_{21}, c_{22}$ into two pairs:

\begin{itemize}
    \item pairs $[c_{11}, c_{12}], [c_{21}, c_{22}]$. In this case, one pair consists of the ``descendants'' of $c_1$, and the other pair consists of the ``descendants'' of $c_2$. We call this a type 1 splitting.
    \item pairs $[c_{11}, c_{21}], [c_{12}, c_{22}]$. In this case both pairs consist of one ``descendant'' of $c_1$ and one ``descendant'' of $c_2$. We call this a type 2 splitting.
\end{itemize}

\begin{lemma} \label{lemma:collision}
    Suppose that a 4-valued dynamics $D$ splits into two 2-valued dynamics $T_1, T_2$. Suppose also that there exists a sequence $(z_n)$ of points in $\C$ converging to $c$ such that splitting of $D$ into $T_1, T_2$ has one type at the points of the sequence but another type at $c$. Then not all images of the point $c$ under the dynamics $D$ are distinct.
\end{lemma}

\begin{proof}
    Assume the converse: suppose that the point $c$ maps to a quadruple of distinct points $[c_{11}, c_{12}, c_{21}, c_{22}]$ under the action of $D$. Choose $\varepsilon > 0$ such that the $\varepsilon$-neighborhoods of points $c_{ij}$ do not intersect, and denote the corresponding neighborhood of the point $[c_{11}, c_{12}, c_{21}, c_{22}] \in \sym^4(\C)$ by $U$. The dynamics $D$ is continuous, therefore the preimage of $U$ under the dynamics $D$ is open. Since it contains $c$, it also contains a $\delta$-neighborhood of $c$. There exists a path lying entirely within $U_\delta(c)$ between any two points within this neighborhood. By construction, the images of all points along this path under the mapping $D$ and hence under the mappings $T_1, T_2$ lie within $U$.
    
    We have chosen the neighborhoods of $c_{11}, c_{12}, c_{21}, c_{22}$ to be non-intersecting, therefore the image of a point under the mapping $T_1$ remains within the same neighborhoods of the points $c_{11}, c_{12}, c_{21}, c_{22}$ as we traverse a path within $U_\delta(c)$. Therefore, the splitting of $D$ into two 2-valued dynamics has the same type throughout the $\delta$-neighborhood of $c$ as at the point $c$. At the same time, almost all points of the sequence $(z_n)$ lie within the neighborhood of $c$. This contradiction proves the lemma.
\end{proof}

\subsection{Main result for dynamics from $\mathcal{P}_2(\C)$}\

Now we prove Theorem \ref{theorem:criteria} in a more explicit form.

\begin{theorem} \label{theorem:exact_criteria}
    Let $T$ be a non-degenerate dynamics in $\mathcal{P}_2(\C)$:
    $$T(z) = -p_1(z) \pm \sqrt{p_1^2(z) - p_0(z)}$$
    Suppose that $p_0$ has a root $z_1$ of an odd multiplicity $d$. Then $T$ cannot be defined by the action of a 2-valued group.
\end{theorem}

We use notation from Subsection 2.2.

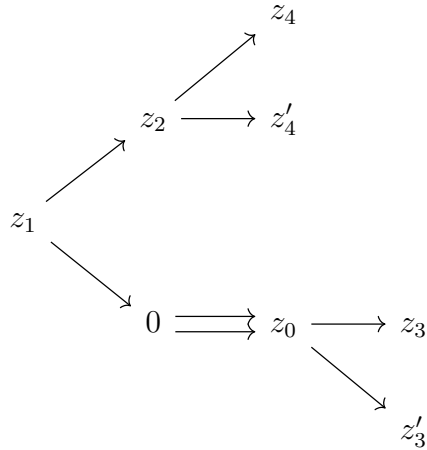
\[
\begin{tikzcd}
    & & z_4 \\
    & z_2 \ar[ru] \ar[r] & z'_4 \\
    z_1 \ar[ru] \ar[rd] \\
    & 0 \ar[r, shift left=1] \ar[r, shift right=1] & z_0 \ar[r] \ar[rd] & z_3 \\
    & & & z'_3
\end{tikzcd}
\]

To prove Theorem \ref{theorem:exact_criteria} we need the following two lemmas.

From Proposition \ref{lemma:4seasons} it follows that in a neighborhood of any $z \in \C$, one can choose a closed curve looping around this point such that the four images of any point on the curve under $T \circ T$ are pairwise distinct. In the proofs of the following two lemmas, we will specifically choose such curves.

\begin{lemma} \label{lemma:z1_area}
    In a neighborhood of the point $z_1$, the 4-valued dynamics $T \circ T$ either does not have a valid splitting into two 2-valued dynamics or has a splitting of the first type.
\end{lemma}
\begin{proof}
Denote  a closed curve looping around $z_1$ by $\gamma_1$. Denote by $g$ a point on this curve, by $g_1, g_2$ images of $g$ under the action of $T$. It follows from \ref{lemma:gamma1} that under a single application of the dynamics $T$ the curve $\gamma_1$ is mapped (1) either to a closed curve making an odd number of turns around $0$ with a pair of images of $g_1$ being swapped when traversing $\gamma_1$, (2) or to a pair of closed curves around $0$ and around $z_2$.

Let us first consider case (1).

\begin{figure}[h]
    \begin{minipage}[h]{0.49\linewidth}
          \begin{center}
              \scalebox{0.7}{
                \begin{tikzpicture}[scale=1] 
                \def \coef {1.42857142857}
                \filldraw [black] (0, 0) circle (2pt) node[anchor=south]{$z_1$};
                \filldraw [black] (\coef * 2, \coef * 0) circle (2pt) node[anchor=south]{$0$};
                \filldraw [black] (\coef * 4, -\coef * 0) circle (2pt) node[anchor=south]{$z_0$};

                \draw[thick, ->] (\coef * 0, \coef * 0) -- (\coef * 2, \coef * 0);
                \draw[thick, ->] (\coef * 2, \coef * 0) -- (\coef * 4, -\coef * 0);

                \filldraw [black] (0, 0.7) circle (1pt) node[anchor=south]{$g$};
                \draw[thick, ->] (0, 0.7) arc (90:450:0.7);
                
                \filldraw [black] (\coef * 2, \coef * 0 + 0.7) circle (1pt) node[anchor=south]{$g_2$};
                \draw[red, thick, ->] (\coef * 2, \coef * 0 + 0.7) arc (90:270:0.7);
                \filldraw [black] (\coef * 2, \coef * 0 - 0.7) circle (1pt) node[anchor=south]{$g_1$};
                \draw[blue, thick, ->] (\coef * 2, \coef * 0 - 0.7) arc (-90:90:0.7);
                
                \filldraw [black] (\coef * 4 + 0.5, \coef * 0 + 0.5) circle (1pt) node[anchor=south]{$g_{21}$};
                \filldraw [black] (\coef * 4 - 0.5, \coef * 0 + 0.5) circle (1pt) node[anchor=south]{$g_{12}$};
                \filldraw [black] (\coef * 4 + 0.5, -\coef * 0 - 0.5) circle (1pt) node[anchor=north]{$g_{11}$};
                \filldraw [black] (\coef * 4 - 0.5, -\coef * 0 - 0.5) circle (1pt) node[anchor=north]{$g_{22}$};
                \draw[red, thick, ->] (\coef * 4 + 0.5, \coef * 0 + 0.5) to[out = 300, in = 60] (\coef * 4 + 0.5, -\coef * 0 - 0.5);
                \draw[red, thick, ->] (\coef * 4 - 0.5, -\coef * 0 - 0.5) to[out = 120, in = 240] (\coef * 4 - 0.5, \coef * 0 + 0.5);
                \draw[blue, thick, ->] (\coef * 4 - 0.5, \coef * 0 + 0.5) to[out = 50, in = 130] (\coef * 4 + 0.5, \coef * 0 + 0.5);
                \draw[blue, thick, ->] (\coef * 4 + 0.5, -\coef * 0 - 0.5) to[out = -130, in = -50] (\coef * 4 - 0.5, -\coef * 0 - 0.5);
                \end{tikzpicture}}
                \end{center}
\end{minipage}
\hfill
\begin{minipage}[h]{0.49\linewidth}
          \begin{center}
              \scalebox{0.7}{
                \begin{tikzpicture}
                \def \coef {1.42857142857}
                \filldraw [black] (0, 0) circle (2pt) node[anchor=south]{$z_1$};
                \filldraw [black] (\coef * 2, \coef * 0) circle (2pt) node[anchor=south]{$0$};
                \filldraw [black] (\coef * 4, -\coef * 0) circle (2pt) node[anchor=south]{$z_3$};

                \draw[thick, ->] (\coef * 0, \coef * 0) -- (\coef * 2, \coef * 0);
                \draw[thick, ->] (\coef * 2, \coef * 0) -- (\coef * 4, -\coef * 0);

                \filldraw [black] (0, 0.7) circle (1pt) node[anchor=south]{$g$};
                \draw[thick, ->] (0, 0.7) arc (90:450:0.7);
                \filldraw [black] (\coef * 2, \coef * 0 + 0.7) circle (1pt) node[anchor=south]{$g_2$};
                \draw[red, thick, ->] (\coef * 2, \coef * 0 + 0.7) arc (90:270:0.7);
                \filldraw [black] (\coef * 2, \coef * 0 - 0.7) circle (1pt) node[anchor=south]{$g_1$};
                \draw[blue, thick, ->] (\coef * 2, \coef * 0 - 0.7) arc (-90:90:0.7);
                
                \filldraw [black] (\coef * 4 + 0.5, \coef * 0 + 0.5) circle (1pt) node[anchor=south]{$g_{21}$};
                \filldraw [black] (\coef * 4 - 0.5, \coef * 0 + 0.5) circle (1pt) node[anchor=south]{$g_{12}$};
                \filldraw [black] (\coef * 4 + 0.5, -\coef * 0 - 0.5) circle (1pt) node[anchor=north]{$g_{11}$};
                \filldraw [black] (\coef * 4 - 0.5, -\coef * 0 - 0.5) circle (1pt) node[anchor=north]{$g_{22}$};
                \draw[red, thick, ->] (\coef * 4 - 0.5, -\coef * 0 - 0.5) to[out = -50, in = -130] (\coef * 4 + 0.5, -\coef * 0 - 0.5);
                \draw[red, ->] (\coef * 4 + 0.5, \coef * 0 + 0.5) to[out = 130, in = 50] (\coef * 4 - 0.5, \coef * 0 + 0.5);
                \draw[blue, thick, ->] (\coef * 4 - 0.5, \coef * 0 + 0.5) to[out = -120, in = 120] (\coef * 4 - 0.5, -\coef * 0 - 0.5);
                \draw[blue, thick, ->] (\coef * 4 + 0.5, -\coef * 0 - 0.5) to[out = 60, in = -60] (\coef * 4 + 0.5, \coef * 0 + 0.5);
                \end{tikzpicture}}
                \end{center}
\end{minipage}
\caption{Case 1: there is no valid splitting in a neighborhood of $z_1$}
\label{paint:2-fail}
\end{figure}
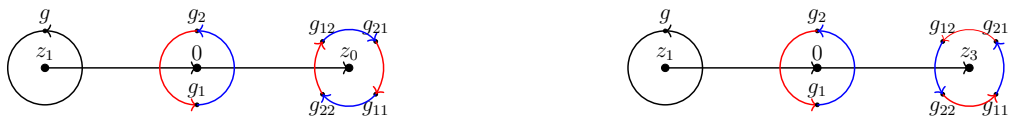

    Denote by $\omega$ the image of the curve $\gamma_1$ under the action of $T$.
    The curve $\omega$ starting at the point $g_1$ makes an odd number of turns around $0$. It follows from Proposition \ref{lemma:omega}, that when traversing $\omega$, the images of any point on this curve are swapped. Therefore, since the images of the point $g$ under $T$ are also swapped when traversing $\gamma_1$, it follows that the four images of $g$ under the double application of $T$ are cyclically permuted when traversing $\gamma_1$. This implies that the 4-valued dynamics $T \circ T$ cannot be split into two continuous 2-valued dynamics, and hence the dynamics $T$ is not defined by the action of a $2$-valued group.

    Now consider case (2). Denote by $\omega$ the branch of $T(\gamma_1)$ in a neighborhood of $0$ making an odd number of turns around this point. It follows from \ref{lemma:gamma1} that such branch exists. Fix a point $g$ on the curve $\gamma_1$. Denote the images of this point under the action of $T$ by $g_1, g_2$, let $g_1$ be on $\gamma_1$. Denote the images of $g_1$ under the action of $T$ by $g_{11}, g_{12}$, the images of $g_2$ under the action of $T$ by $g_{21}, g_{22}$.

    \begin{figure}[h]
        \begin{center}
		          \scalebox{0.7}{
                    \begin{tikzpicture}
                    \def \coef {1.42857142857}
                    \filldraw [black] (0, 0) circle (2pt) node[anchor=south]{$z_1$};
                    \filldraw [black] (\coef * 2, \coef * 1) circle (2pt) node[anchor=south]{$z_2$};
                    \filldraw [black] (\coef * 2, -\coef) circle (2pt) node[anchor=south]{$0$};
                    \filldraw [black] (\coef * 4, \coef * 2.3) circle (2pt) node[anchor=south]{$z_4$};
                    \filldraw [black] (\coef * 4, \coef * 1) circle (2pt) node[anchor=south]{$z'_4$};
                    \filldraw [black] (\coef * 4, -\coef) circle (2pt) node[anchor=south]{$z_0$};

                    \draw[thick, ->] (\coef * 0, \coef * 0) -- (\coef * 2, \coef * 1);
                    \draw[thick, ->] (\coef * 0, \coef * 0) -- (\coef * 2, -\coef * 1);
                    \draw[thick, ->] (\coef * 2, -\coef) -- (\coef * 4, -\coef);

                    \filldraw [black] (0, 0.7) circle (1pt) node[anchor=south]{$g$};
                    \draw[thick, ->] (0, 0.7) arc (90:450:0.7);
                    \filldraw [black] (\coef * 2, \coef * 1 + 0.7) circle (1pt) node[anchor=south]{$g_2$};
                    \draw[red, thick, ->] (\coef * 2, \coef * 1 + 0.7) arc (90:450:0.7);
                    \filldraw [black] (\coef * 2, -\coef + 0.7) circle (1pt) node[anchor=south]{$g_1$};
                    \draw[blue, thick, ->] (\coef * 2, -\coef + 0.7) arc (90:450:0.7);
                    \filldraw [black] (\coef * 4, \coef * 2.3 + 0.7) circle (1pt) node[anchor=south]{$g_{21}$};
                    \filldraw [black] (\coef * 4, \coef * 1 + 0.7) circle (1pt) node[anchor=south]{$g_{22}$};
                    \filldraw [black] (\coef * 4, -\coef + 0.7) circle (1pt) node[anchor=south]{$g_{11}$};
                    \filldraw [black] (\coef * 4, -\coef - 0.7) circle (1pt) node[anchor=south]{$g_{12}$};
                    \draw[blue, thick, ->] (\coef * 4, -\coef + 0.7) arc (90:270:0.7);
                    \draw[blue, thick, ->] (\coef * 4, -\coef - 0.7) arc (-90:90:0.7);
                    \end{tikzpicture}}
                    \end{center}
        \caption{Case 2: splitting of type 1 in a neighborhood of $z_1$}
    \end{figure}
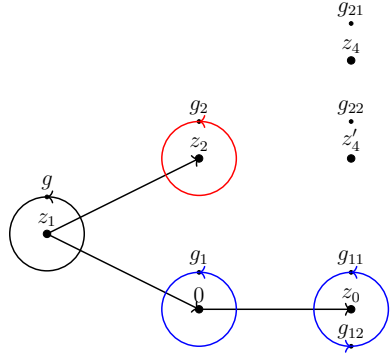

    The curve $\omega$ starting at the point $g_1$ makes an odd number of turns around $0$. Thus, according to Proposition \ref{lemma:omega}, the two images of $g_1$ are swapped when traversing $\omega$.

    All four images of every point on $\gamma_1$ under the action of $T \circ T = [T_1, T_2]$ are distinct, the images of $g_{11}$ and $g_{12}$ are swapped when traversing $\gamma_1$. Therefore since $T_1$ and $T_2$ should be continuous, it follows that $[g_{11}, g_{12}] = T_i(g)$ for some $i$. So the four images of $g$ under the action of $T \circ T$ must split into pairs as follows: $[g_{11}, g_{12}], [g_{21}, g_{22}]$. This is the type 1 splitting.
    \end{proof}

    \begin{lemma} \label{lemma:0_area}
        In a neighborhood of the point $0$ the 4-valued dynamics $T \circ T$ either does not have a valid splitting into two 2-valued dynamics or has a splitting of the second type.
    \end{lemma} 
    \begin{proof}
    Recall that
    $$T \circ T(0) = [z_3, z_3, z'_3, z'_3]$$

    Denote a closed simple curve looping around $0$ by $\omega$, let $g$ be a point on $\omega$.

    Denote the images of $g$ under the action of dynamics $T$ by $g_1, g_2$. It follows from Proposition \ref{lemma:omega}, that $\omega$ maps to a pair of paths from $g_1$ tp $g_2$ and from $g_2$ to $g_1$. Denote these paths by $\gamma_1$ and $\gamma_2$, respectively. Under the application of the dynamics $T$, the point $g_1$ maps to a pair of points $g_{11}, g_{12}$ in neighborhoods of $z_3, z'_3$, respectively, and the point $g_2$ maps to a pair of points $g_{21}, g_{22}$ in neighborhoods of $z_3, z'_3$, respectively.

    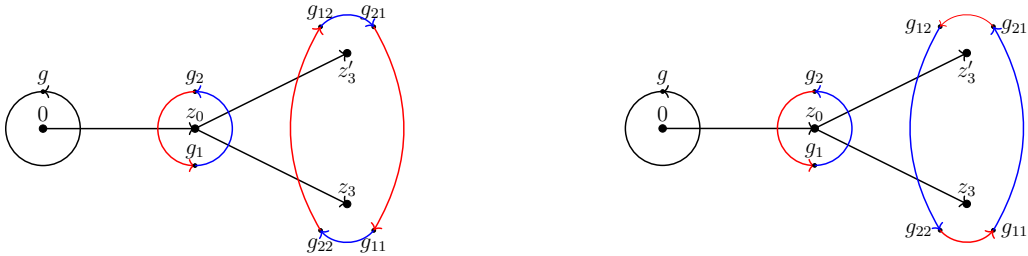
\begin{figure}[h]
        \begin{minipage}[h]{0.49\linewidth}
		      \begin{center}
		          \scalebox{0.7}{
                    \begin{tikzpicture}
                    \def \coef {1.42857142857}
                    \filldraw [black] (0, 0) circle (2pt) node[anchor=south]{$0$};
                    \filldraw [black] (\coef * 2, \coef * 0) circle (2pt) node[anchor=south]{$z_0$};
                    \filldraw [black] (\coef * 4, -\coef * 1) circle (2pt) node[anchor=south]{$z_3$};
                    \filldraw [black] (\coef * 4, \coef * 1) circle (2pt) node[anchor=north]{$z'_3$};

                    \draw[thick, ->] (\coef * 0, \coef * 0) -- (\coef * 2, \coef * 0);
                    \draw[thick, ->] (\coef * 2, \coef * 0) -- (\coef * 4, -\coef * 1);
                    \draw[thick, ->] (\coef * 2, \coef * 0) -- (\coef * 4, \coef * 1);

                    \filldraw [black] (0, 0.7) circle (1pt) node[anchor=south]{$g$};
                    \draw[thick, ->] (0, 0.7) arc (90:450:0.7);
                    
                    \filldraw [black] (\coef * 2, \coef * 0 + 0.7) circle (1pt) node[anchor=south]{$g_2$};
                    \draw[red, thick, ->] (\coef * 2, \coef * 0 + 0.7) arc (90:270:0.7);
                    \filldraw [black] (\coef * 2, \coef * 0 - 0.7) circle (1pt) node[anchor=south]{$g_1$};
                    \draw[blue, thick, ->] (\coef * 2, \coef * 0 - 0.7) arc (-90:90:0.7);
                    
                    \filldraw [black] (\coef * 4 + 0.5, \coef * 1 + 0.5) circle (1pt) node[anchor=south]{$g_{21}$};
                    \filldraw [black] (\coef * 4 - 0.5, \coef * 1 + 0.5) circle (1pt) node[anchor=south]{$g_{12}$};
                    \filldraw [black] (\coef * 4 + 0.5, -\coef - 0.5) circle (1pt) node[anchor=north]{$g_{11}$};
                    \filldraw [black] (\coef * 4 - 0.5, -\coef - 0.5) circle (1pt) node[anchor=north]{$g_{22}$};
                    \draw[red, thick, ->] (\coef * 4 + 0.5, \coef + 0.5) to[out = 300, in = 60] (\coef * 4 + 0.5, -\coef - 0.5);
                    \draw[red, thick, ->] (\coef * 4 - 0.5, -\coef - 0.5) to[out = 120, in = 240] (\coef * 4 - 0.5, \coef * 1 + 0.5);
                    \draw[blue, thick, ->] (\coef * 4 - 0.5, \coef + 0.5) to[out = 50, in = 130] (\coef * 4 + 0.5, \coef * 1 + 0.5);
                    \draw[blue, thick, ->] (\coef * 4 + 0.5, -\coef - 0.5) to[out = -130, in = -50] (\coef * 4 - 0.5, -\coef - 0.5);
                    \end{tikzpicture}}
                    \end{center}
	\end{minipage}
	\hfill
    \begin{minipage}[h]{0.49\linewidth}
		      \begin{center}
		          \scalebox{0.7}{
                    \begin{tikzpicture}
                    \def \coef {1.42857142857}
                    \filldraw [black] (0, 0) circle (2pt) node[anchor=south]{$0$};
                    \filldraw [black] (\coef * 2, \coef * 0) circle (2pt) node[anchor=south]{$z_0$};
                    \filldraw [black] (\coef * 4, -\coef * 1) circle (2pt) node[anchor=south]{$z_3$};
                    \filldraw [black] (\coef * 4, \coef * 1) circle (2pt) node[anchor=north]{$z'_3$};

                    \draw[thick, ->] (\coef * 0, \coef * 0) -- (\coef * 2, \coef * 0);
                    \draw[thick, ->] (\coef * 2, \coef * 0) -- (\coef * 4, -\coef * 1);
                    \draw[thick, ->] (\coef * 2, \coef * 0) -- (\coef * 4, \coef * 1);

                    \filldraw [black] (0, 0.7) circle (1pt) node[anchor=south]{$g$};
                    \draw[thick, ->] (0, 0.7) arc (90:450:0.7);
                    \filldraw [black] (\coef * 2, \coef * 0 + 0.7) circle (1pt) node[anchor=south]{$g_2$};
                    \draw[red, thick, ->] (\coef * 2, \coef * 0 + 0.7) arc (90:270:0.7);
                    \filldraw [black] (\coef * 2, \coef * 0 - 0.7) circle (1pt) node[anchor=south]{$g_1$};
                    \draw[blue, thick, ->] (\coef * 2, \coef * 0 - 0.7) arc (-90:90:0.7);
                    
                    \filldraw [black] (\coef * 4 + 0.5, \coef * 1 + 0.5) circle (1pt) node[anchor=west]{$g_{21}$};
                    \filldraw [black] (\coef * 4 - 0.5, \coef * 1 + 0.5) circle (1pt) node[anchor=east]{$g_{12}$};
                    \filldraw [black] (\coef * 4 + 0.5, -\coef - 0.5) circle (1pt) node[anchor=west]{$g_{11}$};
                    \filldraw [black] (\coef * 4 - 0.5, -\coef - 0.5) circle (1pt) node[anchor=east]{$g_{22}$};
                    \draw[red, thick, ->] (\coef * 4 - 0.5, -\coef - 0.5) to[out = -50, in = -130] (\coef * 4 + 0.5, -\coef - 0.5);
                    \draw[red, ->] (\coef * 4 + 0.5, \coef + 0.5) to[out = 130, in = 50] (\coef * 4 - 0.5, \coef + 0.5);
                    \draw[blue, thick, ->] (\coef * 4 - 0.5, \coef + 0.5) to[out = -120, in = 120] (\coef * 4 - 0.5, -\coef - 0.5);
                    \draw[blue, thick, ->] (\coef * 4 + 0.5, -\coef - 0.5) to[out = 60, in = -60] (\coef * 4 + 0.5, \coef * 1 + 0.5);
                    \end{tikzpicture}}
                    \end{center}
	\end{minipage}
    \caption{Case 1: there is no valid splitting in a neighborhood of $0$}
	\label{paint:2-fail}
    \end{figure}
    
    \begin{figure}[h]
        \begin{minipage}[h]{0.49\linewidth}
		      \begin{center}
		          \scalebox{0.7}{
                    \begin{tikzpicture}
                    \def \coef {1.42857142857}
                    \filldraw [black] (0, 0) circle (2pt) node[anchor=south]{$0$};
                    \filldraw [black] (\coef * 2, \coef * 0) circle (2pt) node[anchor=south]{$z_0$};
                    \filldraw [black] (\coef * 4, -\coef * 1) circle (2pt) node[anchor=north]{$z_3$};
                    \filldraw [black] (\coef * 4, \coef * 1) circle (2pt) node[anchor=south]{$z'_3$};

                    \draw[thick, ->] (\coef * 0, \coef * 0) -- (\coef * 2, \coef * 0);
                    \draw[thick, ->] (\coef * 2, \coef * 0) -- (\coef * 4, -\coef * 1);
                    \draw[thick, ->] (\coef * 2, \coef * 0) -- (\coef * 4, \coef * 1);

                    \filldraw [black] (0, 0.7) circle (1pt) node[anchor=south]{$g$};
                    \draw[thick, ->] (0, 0.7) arc (90:450:0.7);
                    
                    \filldraw [black] (\coef * 2, \coef * 0 + 0.7) circle (1pt) node[anchor=south]{$g_2$};
                    \draw[red, thick, ->] (\coef * 2, \coef * 0 + 0.7) arc (90:270:0.7);
                    \filldraw [black] (\coef * 2, \coef * 0 - 0.7) circle (1pt) node[anchor=south]{$g_1$};
                    \draw[blue, thick, ->] (\coef * 2, \coef * 0 - 0.7) arc (-90:90:0.7);
                    
                    \filldraw [black] (\coef * 4 + 0.5, \coef * 1 + 0.5) circle (1pt) node[anchor=south]{$g_{21}$};
                    \filldraw [black] (\coef * 4 - 0.5, \coef * 1 + 0.5) circle (1pt) node[anchor=south]{$g_{12}$};
                    \filldraw [black] (\coef * 4 + 0.5, -\coef - 0.5) circle (1pt) node[anchor=north]{$g_{11}$};
                    \filldraw [black] (\coef * 4 - 0.5, -\coef - 0.5) circle (1pt) node[anchor=north]{$g_{22}$};
                    \draw[blue, thick, ->] (\coef * 4 + 0.5, -\coef - 0.5) to[out = 60, in = -60] (\coef * 4 + 0.5, \coef * 1 + 0.5);
                    \draw[red, thick, ->] (\coef * 4 + 0.5, \coef + 0.5) to[out = -120, in = 120] (\coef * 4 + 0.5, -\coef - 0.5);
                    \draw[red, thick, ->] (\coef * 4 - 0.5, -\coef - 0.5) to[out = 60, in = -60] (\coef * 4 - 0.5, \coef * 1 + 0.5);
                    \draw[blue, thick, ->] (\coef * 4 - 0.5, \coef + 0.5) to[out = -120, in = 120] (\coef * 4 - 0.5, -\coef - 0.5);
                    \end{tikzpicture}}
                    \end{center}
	\end{minipage}
	\hfill
    \begin{minipage}[h]{0.49\linewidth}
		      \begin{center}
		          \scalebox{0.7}{
                    \begin{tikzpicture}
                    \def \coef {1.42857142857}
                    \filldraw [black] (0, 0) circle (2pt) node[anchor=south]{$0$};
                    \filldraw [black] (\coef * 2, \coef * 0) circle (2pt) node[anchor=south]{$z_0$};
                    \filldraw [black] (\coef * 4, -\coef * 1) circle (2pt) node[anchor=south]{$z_3$};
                    \filldraw [black] (\coef * 4, \coef * 1) circle (2pt) node[anchor=north]{$z'_3$};

                    \draw[thick, ->] (\coef * 0, \coef * 0) -- (\coef * 2, \coef * 0);
                    \draw[thick, ->] (\coef * 2, \coef * 0) -- (\coef * 4, -\coef * 1);
                    \draw[thick, ->] (\coef * 2, \coef * 0) -- (\coef * 4, \coef * 1);

                    \filldraw [black] (0, 0.7) circle (1pt) node[anchor=south]{$g$};
                    \draw[thick, ->] (0, 0.7) arc (90:450:0.7);
                    \filldraw [black] (\coef * 2, \coef * 0 + 0.7) circle (1pt) node[anchor=south]{$g_2$};
                    \draw[red, thick, ->] (\coef * 2, \coef * 0 + 0.7) arc (90:270:0.7);
                    \filldraw [black] (\coef * 2, \coef * 0 - 0.7) circle (1pt) node[anchor=south]{$g_1$};
                    \draw[blue, thick, ->] (\coef * 2, \coef * 0 - 0.7) arc (-90:90:0.7);
                    
                    \filldraw [black] (\coef * 4 + 0.5, \coef * 1 + 0.5) circle (1pt) node[anchor=west]{$g_{21}$};
                    \filldraw [black] (\coef * 4 - 0.5, \coef * 1 + 0.5) circle (1pt) node[anchor=east]{$g_{12}$};
                    \filldraw [black] (\coef * 4 + 0.5, -\coef - 0.5) circle (1pt) node[anchor=west]{$g_{11}$};
                    \filldraw [black] (\coef * 4 - 0.5, -\coef - 0.5) circle (1pt) node[anchor=east]{$g_{22}$};
                    \draw[blue, thick, ->] (\coef * 4 + 0.5, -\coef - 0.5) to[out = 130, in = 50] (\coef * 4 - 0.5, -\coef - 0.5);
                    \draw[red, thick, ->] (\coef * 4 - 0.5, -\coef - 0.5) to[out = -50, in = -130] (\coef * 4 + 0.5, -\coef - 0.5);
                    \draw[blue, thick, ->] (\coef * 4 - 0.5, \coef + 0.5) to[out = -50, in = -130] (\coef * 4 + 0.5, \coef * 1 + 0.5);
                    \draw[red, ->] (\coef * 4 + 0.5, \coef + 0.5) to[out = 130, in = 50] (\coef * 4 - 0.5, \coef + 0.5);
                    \end{tikzpicture}}
                    \end{center}
	\end{minipage}
    \caption{Case 2: splitting of type 2 in a neighborhood of $0$}
	\label{paint:2-pairs}
    \end{figure}
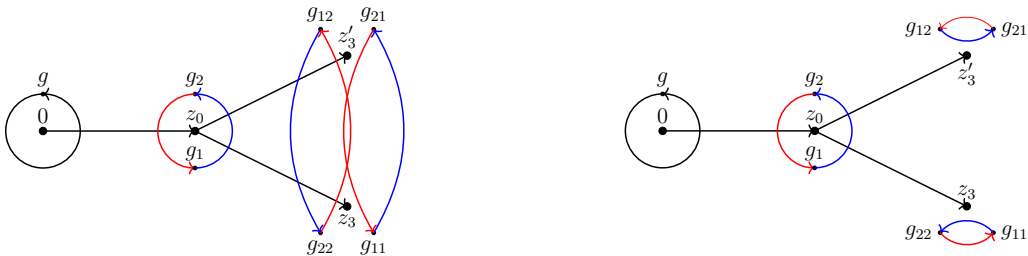

    Since all four images of any point on $\omega$ under the action of $T \circ T$ are distinct, it follows that $g_{11}, g_{12}, g_{21}, g_{22}$ are distinct. The path $\gamma_1$ is mapped to a path from the pair $[g_{11}, g_{12}]$ to the pair $[g_{21}, g_{22}]$, and the path $\gamma_2$ is mapped to a path from the pair $[g_{21}, g_{22}]$ to the pair $[g_{11}, g_{12}]$ under the action of the dynamics $T$.

    Then the image of $\gamma_1$ is either a pair of paths from $g_{11}$ to $g_{21}$ and from $g_{12}$ to $g_{22}$, or a pair of paths from $g_{11}$ to $g_{22}$ and from $g_{12}$ to $g_{21}$.
    Similarly, the image of $\gamma_2$ is either a pair of paths from $g_{21}$ to $g_{11}$ and from $g_{22}$ to $g_{12}$, or a pair of paths from $g_{21}$ to $g_{12}$ and from $g_{22}$ to $g_{11}$.

    Therefore, either the four images $g_{11}, g_{12}, g_{21}, g_{22}$ of the point $g$ under the double application of $T$ are cyclically permuted when traversing $\gamma_1$ (see Fig. \ref{paint:2-fail}) or split in two pairs of swapping points in any way except $[g_{11}, g_{12}], [g_{21}, g_{22}]$ (see Fig. \ref{paint:2-pairs}). In the first case it is impossible to split the dynamics $T \circ T$ into two continuous dynamics because the image of $g$ under both 2-valued dynamics must return to its original position when traversing $\omega$. In the second case there is a type 2 splitting of $T \circ T$ at the point $g$.
    \end{proof}

    Let us proceed to the proof of the Theorem \ref{theorem:exact_criteria}.

\begin{proof}
    If the 4-valued dynamics $T \circ T$ does not have a valid splitting in a neighborhood of any of the points $z_1, 0$, then, according to a Proposition \ref{proposition:4div2}, the dynamics $T$ is not defined by the action of a 2-valued group.

    Otherwise, according to Lemmas \ref{lemma:z1_area}, \ref{lemma:0_area}, we have two points: $\xi_1$ in a neighborhood of $z_1$, $\xi_0$ in a neighborhood of $0$ with the following properties:
    
    1) Either of the points has four pairwise distinct images under $T \circ T$
    
    2) $T \circ T$ has different types of splitting into two continuous two-valued dynamics at the points $\xi_1$ and $\xi_0$. 
    
    Since the dynamics $T$ is non-degenerate, it follows that there exists a path $I:[0, 1] \to \C$ between $\xi_1$ and $\xi_0$ such that at any point $I(t), \; t \in [0, 1]$ the 4-valued dynamics $T \circ T$ has four pairwise distinct images.

    Consider the infimum $t_{inf}$ of the set of points $t \in [0, 1]$ such that the splitting type of $T \circ T$ at the point $I(t)$ does not coincide with the type of splitting of $T \circ T$ at the point $I(0)$. This set is non-empty, as it contains $1$. Then, there exists a monotonically decreasing or monotonically increasing sequence $(t_n)$ of elements of $[0, 1]$ with limit $t_{inf}$ such that $T \circ T$ has one type of splitting at all points $t_n$ and the other type of splitting at the point $t_{inf}$. Since $I(t_n) \to I(t_{inf})$ as $t_n \to t_{inf}$, it follows from Lemma \ref{lemma:collision} that not all images of the point $I(t_{inf})$ under the dynamics $T \circ T$ are distinct. This contradicts the fact that the dynamics $T \circ T$ has four distinct images at every point of the path $I([0,1])$.

    Therefore, there is no valid splitting of $T \circ T$ into two continuous 2-valued dynamics, thus dynamics $T$ cannot be defined by the action of a 2-valued group.
\end{proof}

\subsection{Sufficiency}\ \label{section:adequacy}

Theorem \ref{theorem:criteria} provides a necessary condition for a non-degenerate 2-valued dynamics to be defined by the action of a 2-valued group. This dynamics must have the form
$$z \mapsto -p_1(z) \pm \sqrt{p_1^2(z) - \hat p^2_0(z)}$$

Since $$-p_1(z) \pm \sqrt{p_1^2(z) - \hat p^2_0(z)} = \left(\sqrt{\frac{(-p_1 - \hat p_0)(z)}{2}} \pm \sqrt{\frac{(-p_1 + \hat p_0)(z)}{2}}\right)^2,$$ it follows that the dynamics can be represented as
$$z \mapsto \left(\sqrt{\alpha(z)} \pm \sqrt{\beta(z)}\right)^2,$$

where $\alpha$ and $\beta$ are arbitrary polynomials.

However this condition is not sufficient. There exist 2-valued dynamics of this form such that these dynamics cannot be defined by the action of a 2-valued group.

\begin{prop} 
    The 2-valued dynamics
    $$T(z) = \left(c \pm \sqrt{\gamma(z)}\right)^2, \; c \in \C \setminus \{0\}, $$
    where $\gamma(z)$ has at least two different roots $z_1, z_2$, cannot be defined by the action of a 2-valued group.
\end{prop}

First, let us fix a small $\eps \in \C$. We will determine the precise value of $\eps$ later.

Consider the point $(c + \varepsilon)^2$ in a small neighborhood of the point $c^2$. Under the action of $T$, the points $z \in \C$ such that 
$\gamma(z) = \varepsilon^2$ or $\gamma(z) = (2c+\varepsilon)^2$ map to pairs containing the point $(c + \varepsilon)^2$. The first equation, by the continuity of $\gamma$, has at least one root in a neighborhood of $z_1$ and at least one root in a neighborhood of $z_2$. Denote these roots by $z'_1, z'_2$ respectively. The second equation has at least two roots outside these neighborhoods. Denote them by $z'_3, z'_4$. Note that by slightly perturbing $\varepsilon$ we can ensure that the points $z'_3, z'_4$ be distinct. Each of these points maps to the pair $\left[(c+\varepsilon)^2, (3c+\varepsilon)^2\right]$ under the action of $T$.

If a 2-valued dynamics $T$ can be defined by the action of a 2-valued group $A$ with a generator $a$, then there exists a continuous dynamics $T^{-1}$ defined by the action of the element $\inv(a)$. This dynamics is inverse to the dynamics $T$ in the following sence: each of the 4-valued dynamics $T \circ T^{-1}$ and $T^{-1} \circ T$ splits into a pair of continuous 2-valued dynamics, one of which is $E(z) = [z, z]$.

Since $z_1, z_2$ are mapped to $[c^2, c^2]$ it follows that the point $c^2$ must map to the pair $[z_1, z_2]$ under the action of $T^{-1}$ for $T^{-1} \circ T$ to include $E$. From the continuity of $T^{-1}$, it follows that the image of $(c+\varepsilon)^2$ lies in a neighborhood of the pair $[z_1, z_2] \in \sym^2(\C)$, and thus does not contain either $z'_3$ or $z'_4$.

Therefore $T^{-1} \circ T(z'_3) = \left[z'_1, z'_2, T^{-1}\left((3c+\varepsilon)^2\right)\right]$ and thus, since $z'_1, z'_2 \not= z'_3$, it follows that $T^{-1}((3c+\varepsilon)^2)$ should be equal to $[z'_3, z'_3]$. From similar reasons we obtain $T^{-1}((3c+\varepsilon)^2) = [z'_4, z'_4]$. This contradiction proves the proposition.

\section{Example of a non strongly invertible 2-valued dynamics that can be defined by the action of a 2-valued group}

We prove the theorem \ref{theorem:irreversible} in a more explicit form:
\begin{theorem}
    The 2-valued dynamics $T(z) = (1 \pm  \sqrt{z})^2$ is not strongly invertible, but this dynamics can be defined by the action of a 2-valued group.
\end{theorem}

\begin{proof}
One of the classic examples of 2-valued groups (see, for example, \cite{mval}) called Buchstaber–Novikov 2-valued group is the set $\mathbb{Z}_+$ of non-negative integers where product is defined as follows:
$$n * m = \left[n+m, |n-m|\right]$$

Consider the following set of 2-valued dynamics:
$$\{T_n: \; z \mapsto (n \pm \sqrt{z})^2 \; | \; n \in \mathbb{Z}\}$$

The dynamics $T = T_1$ is defined by the action of the 2-valued group $\mathbb{Z}_+$:
$$\nu(n, z) = T_n(z),$$
because
\begin{multline*}
    \nu(n * m, z) = [\nu(n+m, z), \nu(|n-m|, z)] = \\ [(n+m+\sqrt{z})^2, (n+m-\sqrt{z})^2, (n-m-\sqrt{z})^2, (n-m+\sqrt{z})^2] = \nu(n, \nu(m, z))
\end{multline*}

However, this dynamics is not strongly invertible: not every point has exactly two preimages under the action of $T_1$, taking multiplicities into account. Specifically, the point $0$ has only one simple preimage, $1$.
\end{proof}

\end{document}